\documentclass[12 pt]{article}
\usepackage{amssymb, amsmath, amsfonts, amsthm, graphics,bbm}
\usepackage{mathrsfs}
\usepackage[hmargin=1 in, vmargin = 1 in]{geometry}
\usepackage{tikz-cd} 
\usetikzlibrary{matrix, calc, arrows} 
\usepackage{hyperref}
\usepackage[all]{xy}

\newtheorem{Theorem}{Theorem}
\newtheorem{Lemma}{Lemma}
\newtheorem{Corollary}{Corollary}

\newtheorem{Definition}{Definition}
\newtheorem{Proposition}{Proposition}
\newtheorem{Conjecture}{Conjecture}

\renewcommand{\AA}{\mathbb{A}}

\newcommand{\FF}{\mathbb{F}}

\newcommand{\NN}{\mathbb{N}}

\newcommand{\QQ}{\mathbb{Q}}

\newcommand{\ZZ}{\mathbb{Z}}

\theoremstyle{definition}

\theoremstyle{definition}

\begin{document}

\title{Algebraic Degree Periodicity in Recurrence Sequences}


\author{Daqing Wan\textsuperscript{1}, Hang Yin\textsuperscript{2}}

\maketitle

\begin{center}
\textsuperscript{1}Department of Mathematics, University of California, Irvine, CA 92697\\dwan@math.uci.edu

\bigskip
\textsuperscript{2}Institute of Mathematics, Chinese Academy of Sciences, Beijing \\yy6137@sina.com

\end{center}

\vspace{1cm}

{\bf Abstract}. The degree sequence of the algebraic numbers in an algebraic linear recurrence sequence is 
shown to be virtually periodic. This is proved using the Skolem-Mahler-Lech theorem. It has applications 
to the degree sequence and the minimal polynomial sequence for exponential sums over finite fields. 
The degree periodicity also holds for some more complicated non-linear recurrence sequences. We give one example from 
the iterations of a polynomial map. This depending on the dynamic Mordell-Lang conjecture which 
has been proved in some cases.

\section{Introduction}

This note is motivated by the following number theoretic question 
on algebraic degrees of exponential sums over 
finite fields.

Let $p$ be a prime, $\zeta_p$ be a primitive $p$-th root of unity. 
For each positive integer $k$, let $\FF_{p^k}$ denote the finite field of $p^k$ elements. Let  $f(x_1,x_2,...,x_n)\in\FF_p[x_1,x_2,...,x_n]$ be a polynomial. For each positive integer $k$, we define the $k$-th exponential sum of $f$ to be
$$ S_k(f)=\sum_{(x_1,x_2,...,x_n)\in (\FF_{p^k})^n}
\zeta_p^{{\rm Tr}_k(f(x_1,x_2,...,x_n))}\in \ZZ[\zeta_p],\ k=1,2, \cdots,$$
where ${\rm Tr}_k:\FF_{p^k}\rightarrow\FF_p$ is the absolute trace map.

It is clear that $S_k(f)$ ($k=1,2,\cdots)$ is a sequence of algebraic 
integers in the $p$-th cyclotomic field. As an algebraic integer, its degree $\deg S_k(f)$ 
is a divisor of $p-1$. A natural problem is to understand and compute this 
degree sequence $\deg S_k(f)$ ($k=1,2,\cdots $). The problem is trivial 
if $p=2$. We can assume $p>2$. 
As we shall see, this degree problem is far from being well understood, even in the 
simplest classical case when $f(x)=x^d$ is a one variable monomial. 

A consequence of our main result is the following general structural 
theorem which shows that the degree sequence is virtually periodic. 

\begin{Theorem} The  sequence $\deg S_k(f)$ is virtually periodic in $k$. 
That is, there are positive integers $N$ and $r$ depending on $f$ and $p$ such that for all $k>N$, we have 
$\deg S_{k+r}(f) = \deg S_k(f)$. 
\end{Theorem}

Theoretically, this theorem suggests that the computation of the infinite degree sequence could be done in a finite amount of time. In reality, this 
is far from being so. The problem is that the constant $N$ is not effective, although $r$ (a period) can be made effective. An interesting open problem is to compute 
the degree sequence in effective finite amount of time. 
Equivalently, one would like to have an effective upper bound for the 
total degree of the rational function
$$\sum_{k=1}^{\infty} (\deg S_k(f)) T^k \in \QQ(T).$$
A similar non-effective problem occurs in explicit lower bound for the complex absolute value 
of the sequence $S_k(f)$ as $k$ varies, see Bombieri-Katz \cite{BK10}.

To get some feeling on this degree problem, let us examine the simplest  classical one variable example when $f(x)=x^d$. 
Historically, $S_k(x^d)$ was sometimes called a Gauss sum or a Gauss period in the literature. 
It has been studied extensively, see the survey in Berndt and Evans \cite{BE81}. 
A trivial reduction shows that one can assume that $(d,p)=1$. 
If $p \equiv 1 \mod{d}$, it is known \cite{W20} that 
$$\deg S_k(x^d) = \frac{d}{(d,k)}, \ k=1,2, \cdots .$$
This formula is due to Gauss when $k=1$ and to Myerson \cite{My81} when $(d,k)=1$. 
It follows that the degree sequence $\deg S_k(x^d)$ is periodic (not just virtually periodic) and 
completely determined when $p \equiv 1 \mod{d}$. Another known  
case \cite{W20} is when $(p-1, d)=1$, in which case, $\deg S_k(x^d)=1$ 
for all $k$. Note that $(p^k-1, d)\not =1$ in general, even if $(p-1,d)=1$. 
These results imply that the degree sequence $\deg S_k(x^d)$
($k=1,2,\cdots$) is completely determined for all primes $p$  when $d$ is a prime. If $d$ is not a prime, effectively determining the degree sequence $\deg S_k(x^d)$ for all primes $p$ 
is not known yet. 

Another classical example is the Kloosterman sum. Fix an integer $a$ not divisible by $p$. For 
positive integer $n$, the sequence of $n$-dimensional Kloosterman sums is defined by 
$${\rm Kl}_k(n,a) = \sum_{(x_1,x_2,...,x_n)\in(\FF^*_{p^k})^n}
\zeta_p^{{\rm Tr}_k(x_1+x_2+\cdots +x_n + \frac{a}{x_1\cdots x_n})} \in \ZZ[\zeta_p], \ k=1,2, \cdots .$$
The degree sequence $\deg Kl_k(n,a)$ ($k=1,2,\cdots)$ is again 
far from being well understood. A straightforward Galois theoretic argument shows that 
$$\deg {\rm Kl}_k(n,a) ~ | ~ \frac{p-1}{(n+1, p-1)}, \ k=1,2,\cdots.$$
These two numbers are not expected to be equal in general. 
One simple result in \cite{W95} says that if 
$k$ is not divisible by $p$, then 
$$\deg {\rm Kl}_k(n,a)=\frac{p-1}{(n+1, p-1)}.$$
The special case $n=1$ and $k=1$ goes back to Sali\'e \cite{Sa32}. 
If $p>(2(n+1)^{2k}+1)^2$, this formula also follows from Fisher's 
result \cite{Fi92} on distinctness of Kloosterman sums, proved using $\ell$-adic cohomology. 
Effectively 
computing the degree sequence $\deg {\rm Kl}_k(n,a)$ ($k=1,2,\cdots$) 
is currently unknown in general when $k$ is divisible by $p$.

Our main result of this note is to prove that the degree 
sequence of any linear recurrence sequence in any finite extension of a field of 
characteristic zero is virtually periodic, see Theorem 3 for a precise statement. The proof combines  
Galois theory together with the well known Skolem-Mahler-Lech theorem on zeros of linear recurrence sequences. 
The non-effective part of our result comes from the non-effective 
part of the Skolem-Mahler-Lech theorem. Theorem 1 on 
exponential sums then follows from this and the fact that 
the sequence $S_k(f)$ ($k=1,2,\cdots$) is a linear recurrence sequence of algebraic integers. 
The latter fact is a consequence of the well known Dwork-Bombieri-Grothendick rationality theorem which says 
that the L-function
$$L(f, T)= \exp (\sum_{k=1}^{\infty}\frac{S_k(f)}{k}T^k)$$
is a rational function. 
More generally, the virtual periodicity result stated in Theorem 1 holds for the degree sequence arising from any motivic L-function over finite fields. 
It would be interesting to make the computation of the degree sequence  
effective in various important number theoretic cases, 
in addition to the above exponential sum examples. 

The Skolem-Mahler-Lech theorem is equivalent to a similar theorem on the iteration sequence of a linear map on a vector space. From this dynamic point of view,  
a remarkable generalization is the dynamic Mordell-Lang conjecture for the iteration sequence of a polynomial map, as proposed 
by Ghioca and Tucker \cite{GT09}. This conjecture is still open in general but has been 
proved in some cases. As an illustration, we show that the degree sequence arising from 
the iteration of any one variable polynomial with algebraic coefficients evaluated at an algebraic point is also virtually periodic. This application 
depends on Xie's work \cite{Xi17} on the two variable case of the dynamic Mordell-Lang conjecture.

\section{Degree periodicity and linear recurrence sequences}

We first recall some basic definitions about linear recurrence sequences. 

\begin{Definition}Let $\{a_n\}_{n\geqslant 0}$ be a sequence in a field $K$.
\begin{enumerate}
\item We say $\{a_n\}_{n\geqslant 0}$ is an LRS (linear recurrence sequence) if there exist some $m\in\NN$ and $c_i\in K$ $(1\leqslant i\leqslant m)$ such that 
\[a_{n}=c_1a_{n-1}+\cdots +c_{m-1}a_{n-m+1}+c_ma_{n-m}\]
holds for all $n\geqslant m$.
\item The sequence $\{a_n\}_{n\geqslant 0}$ is called periodic if there exists some integer $r>0$ such that $a_n=a_{n+r}$ holds for all $n\geq 0$.
The sequence $\{a_n\}_{n\geqslant 0}$ is called virtually periodic if there exist some $N,r\in\NN$ such that for all $n>N,a_{n+r}=a_n$.  In both cases, 
such an integer $r$ is called a period of this sequence. 
\item We say $\{a_n\}_{n\geqslant 0}$ is a rational sequence if its 
generating function $f(x)=\sum_na_nx^n$ is a rational function.
\end{enumerate}
\end{Definition}

\begin{Proposition} The following properties are well known and easy to check. 
\begin{enumerate}
\item A virtually periodic sequence is clearly a rational sequence.
\item A sequence $\{a_n\}_{n\geqslant 0}$ is a rational sequence if and only if for some $m\in\ZZ_{\geq 0}$, the new sequence $\{a_n\}_{n\geqslant m}$ is an LRS. Furthermore, an LRS is exactly a rational sequence whose generating function vanishes at $\infty$ i.e. $\sum_na_nx^n=\frac{f(x)}{g(x)}$ with $\deg(f)<\deg(g)$.

\item The sub-sequence $\{ a_{i+nr}\}$ ($n\geq 0$) of any rational sequence $\{a_n\}$ is again a rational 
sequence for all integers $i, r\in \ZZ_{\geq 0}$. 

\item A sequence $\{a_n\}_{n\geqslant 0}$ in an algebraic closure $\bar{K}$ of $K$ is a rational sequence if and only if there are 
finite number of polynomials 
$h_i(x)$ and elements $\beta_i$ over $\bar{K}$ ($1\leq i\leq m$) such that for all 
sufficiently large $n$, we have 
$$a_n = \sum_{i=1}^m h_i(n)\beta_i^n.$$

\end{enumerate}
\end{Proposition}

The last property gives the following lemma which will be used in our proof. 

\begin{Lemma} For each integer $1\leq j \leq \ell$, let $\{a_{jn}\}$ be a rational sequence in $K$. Let $g(x_1,\cdots, x_{\ell})\in K[x_1,\cdots, x_{\ell}]$. Then, the new sequence 
$g(a_{1n}, \cdots, a_{\ell n})$ is a rational sequence in $K$.
\end{Lemma}

We first remind ourselves of the Skolem-Mahler-Lech theorem.

\begin{Theorem}[Skolem-Mahler-Lech]Let $K$ be a field of characteristic 0. Let $\{a_n\}_{n\geqslant 0}$ be a rational sequence with coefficients in $K$, then the set $\{n\in\NN|a_n=0\}$ is virtually periodic. 
\end{Theorem}

We now present our result. 

\begin{Theorem}
Let $\{a_n\}_{n\geqslant 0}$ be a rational sequence with coefficients in a finite extension $L$ of a field $K$ of characteristic zero. We have the following. 
\begin{itemize}
    
\item{(1)}. The sequence $\{\deg(a_n)\}_{n\geqslant 0}$ is virtually periodic, where $\deg(a_n)=[K(a_n):K]$.

\item{(2)}. Let $P_n(T)$ denote the minimal polynomial of $a_n$ over $K$. 
The sequence $\{P_n(T)\}_{n\geqslant 0}$ is a rational sequence, that is, its generating function $\sum_n P_n(T)x^n$ is a rational function in $x$ with poles algebraic over $K$. 
\end{itemize}
\end{Theorem}
\begin{proof}Without loss of generality, we may enlarge L to assume $L/K$ is finite Galois. Let $\sigma\in G={\rm Gal}(L/K)$. Since $\{a_n\}$ is a rational sequence, applying $\sigma$ to it we get another rational sequence $\{\sigma(a_n)\}$. Now the difference $\{\sigma(a_n)-a_n\}$ is also a rational sequence. By the Skolem-Mahler-Lech theorem, we deduce that  $\{n\in\NN|\sigma(a_n)-a_n=0\}$ is virtually periodic. In another word, the indices of those $a_n$ fixed by $\sigma$ is virtually periodic.

For each $\sigma \in G$, let $N_\sigma$ and $r_\sigma$ be positive integers 
such that for all $n>N_\sigma$, $a_n$ is fixed by $\sigma$ if and only if  $a_{n+r_\sigma}$ is fixed by $\sigma$. Let $N=\max_{\sigma\in G}(N_\sigma)$, $r=\prod_{\sigma\in G}r_\sigma$. We see that $N,r$ work for all $\sigma\in G$ uniformly. Now for all $n>N$ and all $\sigma\in G$, $\sigma$ fixes $a_n$ if and only if it fixes $a_{n+r}$, which implies the fixed subgroup $H_n$ of $a_n$ satisfies $H_n=H_{n+r}$. By Galois theory, the degree of $a_n$ is just $[G:H_n]$, hence for $n>N, \deg(a_n)=\deg(a_{n+r})$. This shows that $\{\deg(a_n\}$ is virtually periodic.

Now, for each integer $0\leq i\leq r-1$, we have 
$H_{N+i}=H_{N+i+kr}$ for all integers $k\geq 0$. Let $\sigma_{i1}, \cdots, \sigma_{id_i}$ be a set of representatives for the coset $G/H_{N+i}=G/H_{N+i+kr}$. 
The minimal polynomial of $a_{N+i+kr}$ over $K$ is 
$$P_{N+i+kr}(T) =\prod_{\sigma \in G/H_{N+i+kr}}(T -\sigma (a_{N+i+kr})) = \prod_{j=1}^{d_i}(T -\sigma_{ij}(a_{N+i+kr})).$$
For $\sigma\in G$, the sequence $\sigma (a_{N+i+kr})$ ($k\geq 0$) is an LRS.  
Each coefficient of the above polynomial is an elementary symmetric polynomial of 
finitely many such linear recurrence sequences. It follows that the generating function 
of the subsequence $\{P_{N+i+kr}(T)\}$ ($k\geq 0$) is a rational function in 
$x$ for each $i$. Hence, the generating function of the total 
sequence $\{P_{n}(T)\}$ ($n\geq 0$) 
is also a rational function in $x$. Furthermore, the poles of this 
rational function are clearly algebraic over $K$.  
The proof is complete.

\end{proof}

The following corollary was first suggested by Shaoshi Chen,  based on his computer calculations. 

\begin{Corollary}
Let $\alpha$ be an element in a finite extension $L$ of a field $K$ of characteristic zero. For integer $n\geq 0$, let $P_n(T)$ be the minimal 
polynomial of $\alpha^n$ over $K$. Then, the sequence $P_n(T)$ ($n\geq 0$) is 
a rational sequence, that is, its generating function $\sum_n P_n(T)x^n$ is a rational function in $x$ with poles in the Galois closure of $K(\alpha)$ over $K$. 
\end{Corollary}

Note that this corollary can be proved directly without the Skolem-Mahler-Lech theorem, as 
the set $\{ n | \sigma(\alpha^n) = \alpha^n\}$ is obviously an arithmetic progression. 
This means that the corollary is actually effective. More generally, Theorem 3 is effective 
when the order of the sequence $\{a_n\}$ is at most $2$. This is because that the sequence 
$\{\sigma(a_n) -a_n\}$ then has order at most $4$ and the Skolem-Mahler-Lech theorem 
is effective for LRS with order up to $4$. 

\section{An application to arithmetic dynamics}

The degree periodicity property can hold for sequences of algebraic numbers 
coming from certain non-linear recurrence relations. We summarize the key 
property that makes it work in the following Lemma, 
whose proof  is the same as our proof in section $2$. 

\begin{Lemma} Let $\{ a_n\}$ ($n\geq 0$) be a sequence of elements in a finite 
Galois extension $L$ of a field $K$.  Assume that for 
each $\sigma$ in the Galois group of $L$ over $K$, the set 
$\{n\in\NN|\sigma(a_n)-a_n=0\}$ is virtually periodic. Then the degree 
sequence $\{\deg(a_n)\}_{n\geqslant 0}$ is virtually periodic, where $\deg(a_n)=[K(a_n):K]$.
Furthermore, the sequence of the minimal polynomial of $a_n$ over $K$ is a rational sequence. 
\end{Lemma}

In the case that the sequence is an LRS in a field $L$ of characteristic zero, the condition in the Lemma is satisfied 
by the Skolem-Mahler-Lech theorem and hence one gets the degree periodicity. 
The condition is also satisfied for some more complicated sequences. 
As an illustration, we give one such example coming from arithmetic dynamics. 

\begin{Theorem} Let $L/K$ be a finite extension of number fields. Let $f(x)\in L[x]$ be a polynomial 
and $a\in L$. Then the sequence $\{ \deg_K f^{(n)}(a)\}$ is virtually periodic, where $f^{(n)}$ denotes 
the $n$-th iterate of the polynomial $f$. Furthermore, the sequence of the minimal polynomial 
of $f^{(n)}(a)$ over $K$ is a rational sequence. 
\end{Theorem}

\begin{proof} As before, we may assume that $L$ is a finite Galois extension of $K$ with Galois group $G$. 
For $\sigma\in G$, $\sigma(f)\in L[x]$ is another polynomial over $L$. By the Lemma, 
it is enough to prove that the set 
$$\{ n\in \NN| \sigma(f^{(n)}(a))=f^{(n)}(a)\}=\{ n\in \NN| \sigma(f)^{(n)}(\sigma(a))=f^{(n)}(a)\}$$
is 
virtually periodic for every $\sigma\in G$. Define the polynomial map 
$$F_{\sigma}: \AA_L^2 \longrightarrow \AA_L^2, \ F_{\sigma}(x_1, x_2) = ( f(x_1), \sigma(f)(x_2)).$$
Let $V$ be the diagonal $\{(x, x)\}$ in $\AA^2$. On checks that 
$$\{n\in \NN| \sigma(f^{(n)}(a))=f^{(n)}(a)\} = \{n\in \NN | F_{\sigma}^{(n)}(a, \sigma(a)) \in V\}.$$
The last set is virtually periodic by the affine plane $\AA^2$ case of the dynamic Mordell-Lang conjecture 
as proved in Xie \cite{Xi17}. The proof is complete. 

A natural generalization of the above theorem is the following conjecture. 

\begin{Conjecture} Let $L/K$ be a finite extension of number fields and $m\in \NN$. Let 
$$f=(f_1(x_1,\cdots, x_m), \cdots, f_m(x_1,\cdots, x_m)) \in L[x_1,\cdots, x_m]^m$$
be a polynomial map from $L^m$ to $L^m$. For any element $a=(a_1,\cdots, a_m)\in L^m$, the sequence $\{ \deg_K f^{(n)}(a)\}$ is virtually periodic, where 
$\deg_K f^{(n)}(a)$ denotes the degree of the field extension over $K$ obtained by adjoining the coordinates 
of the point $f^{(n)}(a) \in L^m$. 
\end{Conjecture}

The above theorem shows that this conjecture is true in the case $m=1$. For $m>1$, it 
is a consequence of the affine $\AA^{2m}$ case of the dynamic Mordell-Lang conjecture, which is still 
open.


\end{proof}

\begin{center}
{\bf Acknowledgements}
\end{center}

This note grew out of the Number Theory Online Mini-Course held 
at Xiamen University in August 2020. It is a pleasure to thank 
the many audience, particularly, Shaoshi Chen and his students, for their interests, questions and careful note-taking.

\end{document}